\documentclass[12pt,letterpaper,reqno]{amsart}
\usepackage[margin=1in]{geometry}
\usepackage{amsmath,amsthm,amsfonts,amssymb,amscd,nccmath,mathtools}
\usepackage{amsthm}
\usepackage{enumerate}
\usepackage{mathrsfs}
\usepackage{graphicx}
\usepackage{hyperref}
\usepackage{hhline}
\usepackage{graphbox}
\usepackage[export]{adjustbox}

\mathtoolsset{showonlyrefs=true}
\numberwithin{equation}{section}
\numberwithin{figure}{section}
\numberwithin{table}{section}

\newtheorem{theorem}{Theorem}[section]

\newtheorem{proposition}[theorem]{Proposition}

\newtheorem{lemma}[theorem]{Lemma}

\theoremstyle{definition}

\hypersetup{%
  colorlinks=true,
  linkcolor=blue,
  citecolor=blue,
  linkbordercolor={0 0 1}
}

\newcommand{\N}{\mathbb{N}}
\newcommand{\R}{\mathbb{R}}
\newcommand{\Z}{\mathbb{Z}}
\newcommand{\Q}{\mathbb{Q}}

\newcommand{\ep}{\varepsilon}
\renewcommand{\a}{\alpha}
\newcommand{\D}{\Delta}
\newcommand{\abs}[1]{\left\lvert #1 \right\rvert}

\newcommand{\area}[1]{\text{area } #1}

\newcommand{\Fp}{F^{\langle p \rangle}}
\newcommand{\Bp}{{\mathcal B}^{\langle p \rangle}}
\newcommand{\Cp}{{\mathcal C}^{\langle p \rangle}_m}
\newcommand{\wt}[1]{\widetilde{#1}}

\newcommand{\lrp}[1]{\left(#1\right)}

\title[Non-convex geometry of numbers]{Non-convex geometry of numbers \\ and continued fractions}
\author{Nickolas Andersen}
\address{Brigham Young University Mathematics Department, Provo UT 84602}
\email{nick@math.byu.edu}

\author{William Duke}
\address{UCLA Mathematics Department, Box 951555, Los Angeles, CA 90095-1555}
\email{wdduke@ucla.edu}

\author{Zach Hacking}
\address{Brigham Young University Mathematics Department, Provo UT 84602}
\email{zmhacking@gmail.com}

\author{Amy Woodall}
\address{Brigham Young University Mathematics Department, Provo UT 84602}
\email{amy.woodall713@gmail.com}

\thanks{The second author is supported by NSF grant DMS 1701638 and  Simons Foundation Award 554649.}

\begin{document}

\begin{abstract}
    In recent work, the first two authors constructed a generalized continued fraction called the $p$-continued fraction, characterized by the property that its convergents (a subsequence of the regular convergents) are best approximations with respect to the $L^p$ norm, where $p\geq 1$.
    We extend this construction to the region $0<p<1$, where now the $L^p$ quasinorm is non-convex.
    We prove that the approximation coefficients of the $p$-continued fraction are bounded above by $\frac 1{\sqrt 5}+\ep_p$, where $\ep_p\to 0$ as $p\to 0$.
    In light of Hurwitz's theorem, this upper bound is sharp, in the limit.
    We also measure the maximum number of consecutive regular convergents that are skipped by the $p$-continued fraction.
\end{abstract}

\maketitle

\section{Introduction}

A rational number $r/s$ with $r,s\in \Z$ and $s>0$ is said to be a best approximation to an irrational real number $\a$ if, for all rationals $r'/s'\neq r/s$ with $0<s'<s$, we have
\[
|r-s\a|<|r'-s'\a|.
\]

Lagrange \cite{Lag} showed that each irrational $\a$ has infinitely many best approximations $\{p_n/q_n\}_{n=1}^\infty$ and that for each of them
 \begin{equation}\label{d1}
q_n |p_n-q_n\a|<1.\end{equation}
The best approximations  are given explicitly by the convergents 
\begin{equation}\label{eq:pn-qn-def}
\frac{p_n}{q_n} =b_0 +\frac{1}{b_1+}\,\frac{1}{b_2+}\cdots\frac{1}{b_n}
\end{equation}
of the regular continued fraction expansion of $\a$:
 \begin{equation}\label{scf2}
\a=b_0 +\frac{1}{b_1+}\,\frac{1}{b_2+}\cdots := b_0+\cfrac{1}{b_1 +\cfrac{1}{b_2 +\cfrac{1}{\ddots }}}.
\end{equation}
Classical results of Vahlen and Borel (Theorems~5A and 5B of \cite[Ch.~I]{schmidt}) state that among any successive pair of convergents, there is at least one that satisfies 
$q_n |p_n-q_n\a|<\tfrac{1}{2}$ and among any successive triple, there is at least one that satisfies the Hurwitz bound $q_n |p_n-q_n\a|<\tfrac{1}{\sqrt{5}}$.

The notion of best approximation has the following generalization. 
For $(x,y)\in \R^2$ and a fixed $0< p< \infty$, let
\begin{equation} \label{eq:fp-def}
F^{\langle p\rangle}(x,y)=(|x|^p+|y|^p)^{\frac 1p},
\end{equation}
while $F^{\langle \infty\rangle}(x,y)=\max\{|x|,|y|\}.$
If $p\geq 1$ then $ F^{\langle p\rangle}$  gives a norm on $\R^2$ while if $0<p<1$ it is a quasinorm in that it only satisfies the weakened triangle inequality
\[
F^{\langle p\rangle}(x_1+x_2,y_1+y_2)\leq 2^{\frac{1}{p}-1}(F^{\langle p\rangle}(x_1,y_1)+F^{\langle p\rangle}(x_2,y_2)).
\]
For a fixed $p$, we say that $r/s$ is an $L^p$-best approximation to $\a$ if there is a $t>1$ depending only on $r/s$ so that for any rational $r'/s'\neq r/s$ 
\[
F_t^{\langle p\rangle}(s,r-s\a)<F_t^{\langle p\rangle}(s',r'-s'\a),
\]
where 
\begin{equation}
    F_t^{\langle p\rangle}(x,y)=F^{\langle p\rangle}(t^{-1}x,ty).
\end{equation}
It is not difficult to show that $r/s$ is an $L^\infty$-best approximation to $\a$ if and only if it is a best approximation to $\a$ in Lagrange's sense (see Lemma~6.1 in \cite{AD}).
Generalizing the case $p=1$, which is due to Minkowski \cite{Min1,Min2},  it is shown in \cite{AD} that for any fixed $p\geq 1$ the $L^p$-best approximations to $\a$ are those rationals given by the convergents \[
\frac{r_n}{s_n} =a_0 +\frac{\ep_1}{a_1+}\,\frac{\ep_2}{a_2+}\cdots\frac{\ep_n}{a_n} \qquad (\ep_j=\pm1)
\]
of 
a uniquely determined semi-regular continued fraction expansion of $\a$, the $p$-continued fraction.
Furthermore,  each such best approximation satisfies
\[
F^{\langle p\rangle}(s_n,r_n-s_n\a)<\D_p^{-\frac{1}{2}}
\]
where $\D_p$ is the critical determinant of the unit ball $\mathcal{B}^{\langle p \rangle}=\{(x,y) \in \R^2 : F^{\langle p\rangle}(x,y)<1\}.$ 
The value of $\D_p$ is given in \cite[Section~4]{AD} (see also the references therein). It is increasing on $1\leq p\leq \infty$ with $\D_{1}=\tfrac{1}{2}, \D_{2}=\frac{\sqrt{3}}{2}$ and  $\D_{\infty}=1$.
For any $1\leq p< \infty$,  the  inequality between arithmetic and geometric means  gives  
\begin{equation}\label{convm}
s_n|r_n - s_n \a|\leq \left(\frac{s_n^p + |r_n-s_n\a|^p}{2}\right)^\frac{2}{p}< 4^{-\frac{1}{p}}\D_p^{-1}.
\end{equation}
The right-hand side of this inequality increases from $\frac{1}{2}$ to 1 as $p$ goes from 1 to $\infty$.
The convergents $r_n/s_n$ of the $1$-continued fraction, which is Minkowski's diagonal continued fraction \cite{Min3}, thus satisfy
 \begin{equation}\label{d2}
s_n |r_n-s_n\a|  <\tfrac{1}{2}.\end{equation}
In fact, Minkowski showed that these $r_n/s_n$ coincide with those regular convergents $p_m/q_m$ that satisfy  $q_m |p_m-q_m\a|<\tfrac{1}{2}.$
For any $p\geq1$ the sequence of convergents of the $p$-continued fraction give a subsequence of the regular convergents,
but in general do not give all of those that satisfy $q_m |p_m-q_m\a|< (4^{1/p}\D_p)^{-1}.$
For details and references see \cite{AD}.

In this paper we generalize these results to $L^p$-best approximations where $p<1$.
In view of the above, our prime motivation is to show that the right-hand side  in the generalization of \eqref{convm} for these approximations can be as close to  $\tfrac{1}{\sqrt{5}}$ as desired.
This necessitates letting $p \rightarrow 0.$
The associated $p$-continued fractions 
 \begin{equation}\label{scf3}
\a=a_0 +\frac{\ep_1}{a_1+}\,\frac{\ep_2}{a_2+}\cdots, \qquad \gcd(\ep_n,a_n,\ep_{n+1})=1, 
\end{equation}
that we construct have the property that for a fixed $p<1$ the partial numerators $\ep_n$ are integers with
$
|\ep_n|\leq M_p 
$
where $M_p$ depends only on $p$.   However,  $M_p \rightarrow \infty$  as $p\rightarrow 0.$
Since now $F^{\langle p\rangle}$ is only a quasinorm, the ball $\mathcal{B}^{\langle p \rangle}$ is not convex and  we must apply some results of Mordell and Watson \cite{mordell,watson} from the non-convex geometry of numbers (see also the book of Gruber and Lekkerkerker \cite{gruber} and the references therein). 
The theorem below summarizes the main properties of the $p$-continued fraction.

\begin{theorem} \label{thm:main}
    For each $\delta>0$ there exists a $p = p_\delta\in (0,1)$ such that for any irrational $\alpha$ the $p$-continued fraction \eqref{scf3} of $\alpha$ has the following properties.
    \begin{enumerate}
        \item The convergents are precisely the best approximations to $\alpha$ with respect to $\Fp$.
        \item Each convergent $r_n/s_n$ satisfies $s_n|r_n-s_n\alpha|< \frac{1}{\sqrt 5}+\delta$.
        \item 
        There exists a constant $M_p$ (depending only on $p$) such that $|\ep_n|\leq M_p$ for all $n$.
    \end{enumerate}
\end{theorem}

For any $p>0$ the $p$-convergents of $\alpha$ form a subsequence of the regular convergents.
Computation confirms that more regular convergents are skipped by the $p$-continued fraction than just those with $s|r - s\a|$ large.
We now describe the relationship between the $p$-convergents $r_n/s_n$ and the regular convergents $p_n/q_n$.

\begin{theorem} \label{thm:max-ell}
    Fix $p\in (0,1)$ and an irrational $\alpha\in \R$.
    Let $n$ be a positive integer and let $m$ and $\ell$ be such that $r_n/s_n=p_m/q_m$ and $r_{n+1}/s_{n+1}=p_{m+\ell}/q_{m+\ell}$.
    Then 
    \begin{equation} \label{eq:ell-upper-bound}
        \ell \leq \log_\varphi\left( 2^{4/p-1}\sqrt 5 \, \frac{\Gamma\big(1+\frac1p\big)^2}{\Gamma\big(1+\frac2p\big)}\right),
    \end{equation}
    where $\varphi = \frac 12(1+\sqrt 5)$.
\end{theorem}

We now show that as $p$ tends to zero, arbitrarily many regular convergents can be skipped in the $p$-continued fraction. This is the case for every irrational for which the regular continued fraction is $1$-periodic.
See Table~\ref{tab:phi-conv} for the case $\alpha=\frac 12(1+\sqrt 5)$.

\begin{small}\begin{table}[b]
    \centering
    \def\arraystretch{0.5}
    \setlength\tabcolsep{7.0pt}
    
    \begin{tabular}{r|cccccccccc}
         $p$\phantom{1} & \multicolumn{10}{c}{First 10 convergents of the $p$-continued fraction}
         \\
         \hhline{===========}
         \\
         $\infty$ & $2$ & $\mfrac{3}{2}$ & $\mfrac{5}{3}$ & $\mfrac{8}{5}$ & $\mfrac{13}{8}$ & $\mfrac{21}{13}$ & $\mfrac{34}{21}$ & $\mfrac{55}{34}$ & $\mfrac{89}{55}$ & $\mfrac{144}{89}$
         \\ \\
         \hline
         \\
         $0.5$ & $2$ & $\mfrac{5}{3}$ & $\mfrac{8}{5}$ & $\mfrac{13}{8}$ & $\mfrac{21}{13}$ & $\mfrac{34}{21}$ & $\mfrac{55}{34}$ & $\mfrac{89}{55}$ & $\mfrac{144}{89}$ & $\mfrac{233}{144}$
         \\ \\
         \hline
         \\
         $0.4$ & $\mfrac{5}{3}$ & $\mfrac{8}{5}$ & $\mfrac{13}{8}$ & $\mfrac{21}{13}$ & $\mfrac{34}{21}$ & $\mfrac{55}{34}$ & $\mfrac{89}{55}$ & $\mfrac{144}{89}$ & $\mfrac{233}{144}$ & $\mfrac{377}{233}$
         \\ \\
         \hline
         \\
         $0.3$ & $\mfrac{8}{5}$ & $\mfrac{13}{8}$ & $\mfrac{21}{13}$ & $\mfrac{34}{21}$ & $\mfrac{55}{34}$ & $\mfrac{89}{55}$ & $\mfrac{144}{89}$ & $\mfrac{233}{144}$ & $\mfrac{377}{233}$ & $\mfrac{610}{377}$
         \\ \\
         \hline
         \\
         $0.25$ & $\mfrac{13}{8}$ & $\mfrac{21}{13}$ & $\mfrac{34}{21}$ & $\mfrac{55}{34}$ & $\mfrac{89}{55}$ & $\mfrac{144}{89}$ & $\mfrac{233}{144}$ & $\mfrac{377}{233}$ & $\mfrac{610}{377}$ & $\mfrac{987}{610}$
         \\ \\
         \hline
         \\
         $0.2$ & $\mfrac{21}{13}$ & $\mfrac{34}{21}$ & $\mfrac{55}{34}$ & $\mfrac{89}{55}$ & $\mfrac{144}{89}$ & $\mfrac{233}{144}$ & $\mfrac{377}{233}$ & $\mfrac{610}{377}$ & $\mfrac{987}{610}$ & $\mfrac{1597}{987}$
         \\ \\
         \hline
         \\
         $0.18$ & $\mfrac{34}{21}$ & $\mfrac{55}{34}$ & $\mfrac{89}{55}$ & $\mfrac{144}{89}$ & $\mfrac{233}{144}$ & $\mfrac{377}{233}$ & $\mfrac{610}{377}$ & $\mfrac{987}{610}$ & $\mfrac{1597}{987}$ & $\mfrac{2584}{1597}$
         \\ \\
         \hline
         \\
         $0.16$ & $\mfrac{55}{34}$ & $\mfrac{89}{55}$ & $\mfrac{144}{89}$ & $\mfrac{233}{144}$ & $\mfrac{377}{233}$ & $\mfrac{610}{377}$ & $\mfrac{987}{610}$ & $\mfrac{1597}{987}$ & $\mfrac{2584}{1597}$ & $\mfrac{4181}{2584}$
         \\ & & &
    \end{tabular}
    \caption{Selected values of $p$ with the first several convergents of the $p$-continued fraction for $\varphi=\frac 12(1+\sqrt 5)$.}
    \label{tab:phi-conv}
\end{table}\end{small}

\begin{theorem} \label{thm:1-per-big-skips}
   Let $a\geq 1$ be an integer and let $\alpha = a+\frac{1}{a+}\,\frac{1}{a+}\,\frac{1}{a+}\cdots=\frac 12(a+\sqrt{a^2+4})$.
   Then for each sufficiently large positive integer $m$ there exists a $p\in (0,1)$ such that $r_0/s_0 = p_m/q_m$.
\end{theorem}

The paper is organized as follows.
In Section~\ref{sec:mordell} we review Mordell's result on the geometry of numbers for the quasinorms $\Fp$.
Section~\ref{sec:algorithm} contains the algorithm that generates the $p$-continued fraction and the proof of Theorem~\ref{thm:main}.
Sections~\ref{sec:skipped} and \ref{sec:1-per} address the question of how many of the regular convergents can be skipped by the $p$-continued fraction; they contain the proofs of Theorems~\ref{thm:max-ell} and \ref{thm:1-per-big-skips}.

In Sections~\ref{sec:mordell}, \ref{sec:algorithm}, and \ref{sec:skipped}, the value of $p\in (0,1)$ is fixed, so in those sections we will usually suppress the dependence on $p$ from the notation.
In Section~\ref{sec:1-per}, the value of $p$ is allowed to vary; there, the notation will reflect the dependence on $p$.

\section{Geometry of numbers for non-convex bodies}
\label{sec:mordell}

Let $F=\Fp$, defined in \eqref{eq:fp-def}.
When $p\geq 1$, $F$ defines a norm on $\R^2$, but when $p\in (0,1)$, $F$ is not a norm because the triangle inequality does not hold.
However, for such $p$ we have the following quasi-triangle inequality, so $F$ defines a quasinorm.
\begin{lemma} Let $p\in (0,1)$. Then for any $(x_1,y_1), (x_2,y_2) \in \R^2$, we have
\begin{equation}
    F(x_1+x_2,y_1+y_2) \leq 2^{\frac 1p-1} \left(F(x_1,y_1) + F(x_2,y_2)\right).
\end{equation}
\end{lemma}
\begin{proof}
We start by proving that $|a+b|^p \leq |a|^p + |b|^p$ for any $a,b\in \R$.
By homogeneity and the usual triangle inequality, it suffices to prove that $(1+t)^p \leq 1+t^p$ for $t\geq 0$, but this follows from the fact that the derivative of $t\mapsto (1+t)^p-t^p-1$ is nonpositive when $p<1$.
    
Using that $|a+b|^p \leq |a|^p + |b|^p$ we obtain
\begin{equation}
    F(x_1+x_2,y_1+y_2) \leq \left(F(x_1,y_1)^p + F(x_2,y_2)^p\right)^{\frac 1p}.
\end{equation}
Now since $t\mapsto t^{1/p}$ is convex we have $(\frac{a+b}{2})^{1/p} \leq \frac{a^{1/p}+b^{1/p}}{2}$, from which it follows that
\begin{equation}
    \left(F(x_1,y_1)^p + F(x_2,y_2)^p\right)^{\frac 1p} \leq 2^{\frac 1p-1} \left(F(x_1,y_1) + F(x_2,y_2)\right),
\end{equation}
as desired.
\end{proof}

Our starting point for studying best approximations with respect to the quasinorm $F$ is the following theorem of Mordell \cite[Section~8]{mordell}, which we have updated to reflect later work of Watson \cite{watson}.
If $L\subseteq \R^2$ is a full lattice, the determinant of $L$ equals $|\!\det g|$, where $g\in \operatorname{GL}_2(\R)$ is any matrix whose rows form a $\Z$-basis for $L$.

\begin{theorem}[Mordell and Watson] \label{thm:mordell}
    Let $p\in (0,1)$ and
    let $(a_p,b_p)$ be the unique solution to the equations
    \begin{align}
        (a+b)^p + (a-b)^p &= a^p + b^p, \label{eq:a-b-sys-1} \\
        a^2 + b^2 &= 2, \label{eq:a-b-sys-2}
    \end{align}
    with $a_p>b_p$.
    Define $c_p$ by
    \begin{equation} \label{eq:cp-def}
        c_p = 2^{-\frac 1p} F(a_p,b_p).
    \end{equation}
    Then every full lattice has a point $(x,y)\neq (0,0)$ satisfying
    \begin{equation} \label{eq:mordell-ineq}
        F(x,y) \leq 2^{\frac 1p-\frac 12} (\det L)^{\frac 12} c_p.
    \end{equation}
    Furthermore, for $p\in (0.3295\ldots, 1)$ the inequality is sharp.
\end{theorem}

Note that the maximum value of the function $(a,b)\mapsto a^p + b^p$ subject to the constraint $a^2+b^2=2$ is $2$, and this occurs only when $a=b=1$.
Thus $c_p<1$ for all $p<1$.
Also, clearly $a_p$ and $b_p$ both approach $1$ as $p$ tends to $1$, so we have
\begin{equation} \label{eq:cp-limit-1}
    \lim_{p\to 1^-}c_p = 1.
\end{equation}

Applying Theorem~\ref{thm:mordell} to the lattice $(0,t)\Z + (t^{-1},-\alpha t)\Z$,
we find that for every $t\geq 1$ there exists a rational $r/s$ with $s>0$ such that
\begin{equation} \label{eq:Ft-leq-cp2}
    F_t(s,r-s\a) \leq 2^{\frac 1p-\frac12} c_p.
\end{equation}
Define
\begin{equation} \label{eq:beta-cp}
    \beta_p = 2^{\frac 2p-1}c_p^2.
\end{equation}
Following Mordell, if we let $j_p = b_p/a_p<1$ then
\begin{equation} \label{eq:ap-bp-jp}
    a_p = \sqrt{\frac{2}{j_p^2+1}} \quad \text{ and } \quad b_p = j_p\sqrt{\frac{2}{j_p^2+1}}
\end{equation}
and from this we find that
\begin{equation} \label{eq:cp-beta}
    \beta_p = \frac{(1+j_p^p)^{2/p}}{1+j_p^2}.
\end{equation}

Applying the inequality between arithmetic and geometric means and \eqref{eq:Ft-leq-cp2}--\eqref{eq:cp-beta}, we find that for each $t\geq 1$ there exists a rational $r/s$ with $s>0$ such that
\begin{equation} \label{eq:am-gm}
    s\abs{r-s\a} \leq 4^{-\frac 1p}\left[F_t(s,r-s\a)\right]^2 \leq 4^{-\frac{1}{p}}\beta_p,
\end{equation}
where equality holds for the first inequality if and only if $s=t^2\abs{r-s\alpha}$.
Letting $t\to\infty$ we find that 
\begin{equation} \label{eq:am-gm-beta}
    s\abs{r-s\a} < 4^{-\frac 1p} \beta_p
\end{equation}
for infinitely many rational approximations $r/s$.
The proposition below describes the values of $4^{-1/p}\beta_p$ for $p\in (0,1)$.

\begin{proposition} \label{thm:beta}
    Let $p\in (0,1)$.
    Then
    \begin{enumerate}
        \item $\lim\limits_{p\to 1^-} 4^{-1/p}\beta_p = \frac 12$,
        \item $\lim\limits_{p\to 0^+} 4^{-1/p}\beta_p = \frac 1{\sqrt 5}$, and
        \item the function $p\mapsto 4^{-1/p}\beta_p$ is increasing on $(0,1)$.
    \end{enumerate}
\end{proposition}

\begin{figure}
    \centering
    \includegraphics{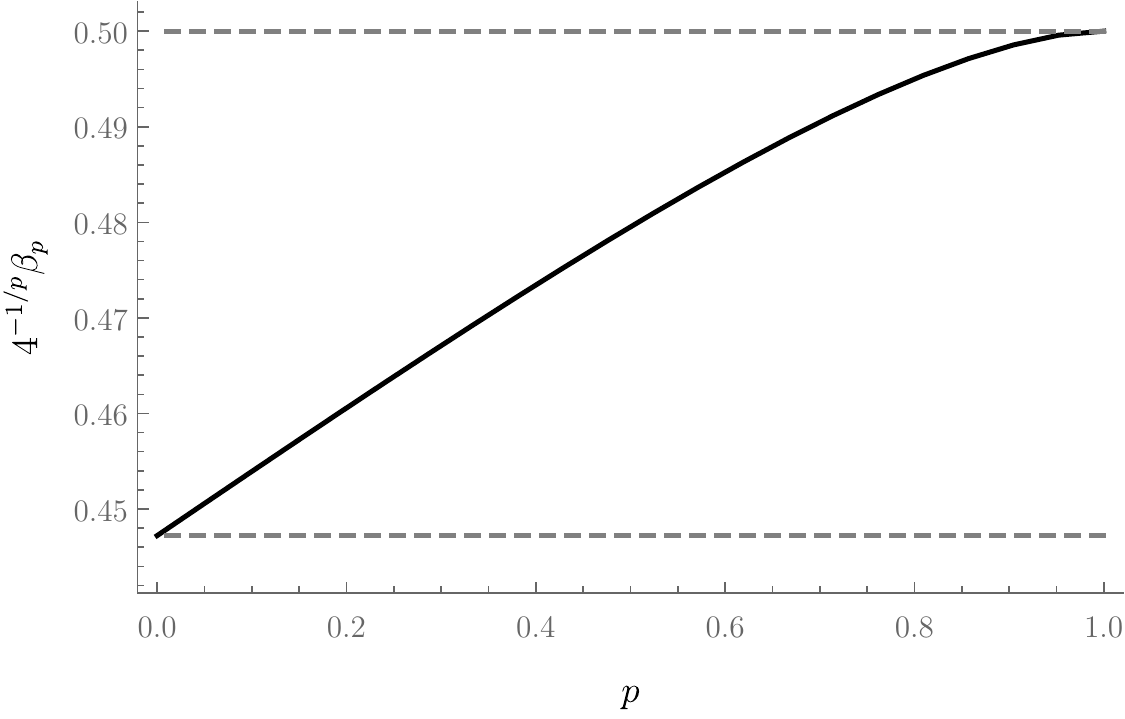}
    \caption{A plot of $4^{-1/p} \beta_p$ for $p \in (0,1)$, with the limiting values of $\frac{1}{2}$ and $\frac{1}{\sqrt{5}}$ shown as dashed lines.}
    \label{fig:beta_p}
\end{figure}

In order to study the behavior of $\beta_p$ it is apparent that we must understand the behavior of $j_p$.
The next lemma follows from Section~2 of \cite{watson}; its proof is elementary but quite tedious.

\begin{lemma} \label{lem:jp}
The function $p\mapsto j_p$ is differentiable and strictly increasing on $(0,1)$ and
\begin{equation} \label{eq:jp-limit}
    \lim_{p\to 0^+} j_p = \mfrac{-1+\sqrt 5}{2}.
\end{equation}
\end{lemma}

Using Lemma~\ref{lem:jp} the proof of Proposition~\ref{thm:beta} is relatively straightforward.

\begin{proof}[Proof of Proposition~\ref{thm:beta}]
(1) By \eqref{eq:beta-cp} and \eqref{eq:cp-limit-1} we have
    \begin{equation}
        \lim_{p\to 1^-} 4^{-\frac 1p}\beta_p = \lim_{p\to 1^-}\mfrac 12c_p^2 = \mfrac 12.
    \end{equation}
    
(2) As $p$ tends to zero we have
\begin{equation}
    4^{-\frac 1p} \frac{(1+x^p)^{2/p}}{1+x^2} = \frac{1}{x^{-1}+x} + O(p)
\end{equation}
uniformly for $x\in [\frac 12,1]$.
By \eqref{eq:jp-limit} we compute
\begin{equation}
    \lim_{p\to 0^+} 4^{-\frac1p}\beta_p = \lim_{p\to 0^+} 4^{-\frac 1p} \frac{(1+j_p^p)^{2/p}}{1+j_p^2} = \frac{1}{2\lrp{-1+\sqrt 5}^{-1} + \frac{1}{2}\lrp{-1+\sqrt 5}} = \frac{1}{\sqrt 5}.
\end{equation}

(3) We will prove that the derivative of the function $p\mapsto 4^{-1/p}\beta_p$ is positive.
Since $\beta_p>0$ it is enough to prove that $f'(p)>0$, where
\begin{equation}
    f(p) = \log(4^{-1/p}\beta_p) = -\frac{2\log 2}{p} + \frac{2\log(1+j_p^p)}{p} - \log(1+j_p^2).
\end{equation}
A straightforward computation yields
\begin{equation}
    f'(p) = \frac{2}{p^2} \left( \log 2 - \log(1+j_p^p) - \frac{j_p^p \log(1/j_p^p)}{1+j_p^p}\right) + 2j_p'\left(\frac{j_p^{p-1}}{1+j_p^p} - \frac{j_p}{1+j_p^2}\right).
\end{equation}
By Lemma~\ref{lem:jp} we have $j_p'>0$.
Since $j_p\in (0,1)$ and $p\in (0,1)$ we have $j_p^{p-2}>1$ so
\begin{equation}
    1+j_p^p < j_p^{p-2} + j_p^p = j_p^{p-2}(1+j_p^2),
\end{equation}
which implies
\begin{equation}
    \frac{j_p^{p-1}}{1+j_p^p} - \frac{j_p}{1+j_p^2}>0.
\end{equation}
Thus to show that $f'(p)>0$ it suffices to prove that $g(j_p^p) < \log 2$, where
\begin{equation}
    g(x) := \log(1+x) + \frac{x \log(1/x)}{1+x}.
\end{equation}
Indeed, we have $g(x) < \log 2$ for any $x\in (0,1)$ because 
\begin{align}
    g'(x) = \frac{\log(1/x)}{(1+x)^2}>0,
\end{align}
and $\lim\limits_{x\to 1^-}g(x) = \log 2$.
\end{proof}

\section{The \texorpdfstring{$p$}{p}-continued fraction for \texorpdfstring{$p\in (0,1)$}{p in (0,1)}}
\label{sec:algorithm}

In this section we prove Theorem~\ref{thm:main}. Fix $p \in (0, 1)$.
We will need the following analogue of Minkowski's first convex body theorem for the non-convex bodies
\begin{equation}
    \mathcal B_t(P) = \left\{ P'\in \R^2 : F_t(P') < F_t(P) \right\},
\end{equation}
where $P\in \R^2$ and $t>0$.
This is essentially an immediate corollary of Mordell and Watson's Theorem~\ref{thm:mordell}.
We say that a lattice is admissible for a set $S$ if the only lattice point inside $S$ is the origin.

\begin{proposition}\label{prop:max-area}
    Fix $p\in (0,1)$, $t \geq 1$, and $P\in \R^2$.
    If $L$ is an admissible lattice for $\mathcal B_t(P)$ then
    \begin{equation}
        \operatorname{area} \mathcal B_t(P) \leq C_p \operatorname{det} L,
    \end{equation}
    where
    \begin{equation}
        C_p = 2^{\frac 2p+1}\frac{\Gamma\big(1+\frac 1p\big)^2}{\Gamma\big(1+\frac 2p\big)} c_p^2
    \end{equation}
    and $c_p$ is given in \eqref{eq:cp-def}.
\end{proposition}

\begin{proof}
Let $R=F_t(P)$.
By Theorem~\ref{thm:mordell}, there is a point $(x,y)\in L$ that satisfies \eqref{eq:mordell-ineq}.
But this point is not in $\mathcal B_t(P)$ since $L$ is admissible for $\mathcal B_t(P)$.
Thus
\begin{equation} \label{eq:R-ineq}
    R \leq F(x,y) \leq  2^{\frac 1p-\frac 12} c_p (\operatorname{det} L)^{\frac 12} .
\end{equation}
The area of $\mathcal B_t(P)$ is given by
\begin{equation}
    \operatorname{area} \mathcal B_t(P) = 4\int_0^{R} (R^p-x^p)^{\frac1p} \, dx = 4R^2 \int_0^1 (1-x^p)^{\frac1p} \, dx = 4R^2 \frac{\Gamma\big(1+\frac1p\big)^2}{\Gamma\big(1+\frac2p\big)}.
\end{equation}
From this computation and \eqref{eq:R-ineq} we have
\begin{equation}
    \operatorname{area} \mathcal B_t(P) \leq 2^{\frac 2p+1}\frac{\Gamma\big( 1+\frac 1p \big)^2}{\Gamma\big(1+\frac 2p\big)} c_p^2 (\operatorname{det} L),
\end{equation}
as desired.
\end{proof}

We will generate the convergents of the $p$-continued fraction using a recursive algorithm, described in Lemma \ref{lem:algorithm}, that yields a sequence of points in the lattice
\[
    L_\alpha = (0,1)\Z+(1,-\alpha)\Z,
\]
where a point $(s,r-s\a)\in L_\alpha$ corresponds to a rational $r/s$.
Of course we can assume without loss of generality that $\alpha\in (-\frac 12,\frac 12)$.
The properties of the algorithm will show that the sequence of convergents is, in fact, convergent and is precisely the sequence of best approximations with respect to $F$, ordered by increasing denominator.

We will use this elementary lemma several times.

\begin{lemma}\label{lem:lower-norm}
    Fix $p \in (0,1)$.
    If $\abs{x'} \leq \abs{x}$ and $\abs{y'} \leq \abs{y}$, then $F(x', y') \leq F(x, y)$.
\end{lemma}

We also have the following property of $\mathcal{B}_t(P)$ as $t$ varies. Roughly speaking, this lemma shows that in the first quadrant, $\mathcal B_t(P)$ contracts to the left of $P$ and expands to the right of $P$ as $t$ increases.

\begin{lemma}\label{lemma:incr-decr-t}
    Fix $p\in (0,1)$ and $P = (x_P, y_P)\in \R^2$ with $x_P \geq 0$, and let $t_0, t_1 \in \R$ with $t_1 > t_0 > 0$. 
    Let $x>0$ and for $j\in \{0,1\}$ define $y_j>0$ by $F_{t_j}(x,y_j)=F_{t_j}(P)$, when a solution to this equation exists. We have that
    \begin{enumerate}
        \item if $x<x_P$, then $y_0>y_1$, and
        \item if $x>x_P$, then $y_0<y_1$.
    \end{enumerate}
\end{lemma}

\begin{proof}

From the implicit definition of $y_j$, we find \[y_j = t_j^{-2}\left(x_P^p - x^p + t_j^{2p}\abs{y_P}^p\right)^\frac{1}{p}\] as an explicit formula for $y_j$. 

For algebraic ease, we consider the quantity $y_0^p-y_1^p$. Since $p>0$, this quantity will have the same sign as $y_0-y_1$. Thence
\begin{align*}
    y_0^p - y_1^p
    &= t_0^{-2p}\left(x_P^p - x^p + t_0^{2p}\abs{y_P}^p\right) - t_1^{-2p}\left(x_P^p - x^p + t_1^{2p}\abs{y_P}^p\right)
    \\
    &= t_0^{-2p}(x_P^p - x^p) + \abs{y_P}^p - t_1^{-2p}(x_P^p - x^p) - \abs{y_P}^p
    \\
    &= (t_0^{-2p} - t_1^{-2p})(x_P^p - x^p).
\end{align*}
Note that the first term in the above expression is positive, since $t_1 > t_0$. If $x < x_P$, then the second term in the above expression is positive, meaning $y_0^p - y_1^p$ is positive and $y_0 > y_1$. If instead $x > x_P$, the argument reverses and $y_0 < y_1$.
\end{proof}

The lemma below describes the aforementioned recursive algorithm.

\begin{lemma}\label{lem:algorithm}
    Fix an irrational $\alpha\in (-\frac12, \frac12)$ and set $P_{-1} = (0,1)$ and $t_{-1} = 1$. For each $m \in \Z$ with $m \geq 0$, there exists a $P_m=(x_m,y_m)\in L_\alpha$ and a $t_m > 1$ with the following properties:
    \begin{enumerate}
        \item $x_m > x_{m-1}$ and $\abs{y_m} < \abs{y_{m-1}}$,
        \item $L_\alpha$ is admissible for $\mathcal B_{t_m}(P_m)$,
        \item $F_{t_m}(P_{m-1}) = F_{t_m}(P_{m})$, and
        \item for any $t \in (t_{m-1}, t_m)$, there is no $P' = (x', y') \in L_\alpha$ different from $P_{m-1}$ with $x' > 0$ and $F_t(P') \leq F_t(P_{m-1})$.
    \end{enumerate}
\end{lemma}

\begin{proof}

    We construct the points $P_m$ inductively.
    By the construction of $L_\alpha$, the only points on the unit ball $\mathcal B = \mathcal B_{t_{-1}}(P_{-1})$ are $\pm P_{-1}$.
    For any $P=(x,y)$ with $y\neq 0$ we have
    \begin{equation}
        F_t(P) = F(t^{-1}x, ty) \geq F(0,ty) = t\abs{y},
    \end{equation}
    so $F_t(P)$  increases as $t$ increases.
    By Proposition $\ref{prop:max-area}$, there must be a maximum $t$ for which $L_\alpha$ is admissible for $\mathcal B_t(P_{-1})$.
    Call this maximum $t_0$.
    Among the finitely many $P' = (x', y') \in L_\alpha$ with $F_{t_{0}}(P') = F_{t_{0}}(P_{-1})$, there is a unique $P'$ with maximal $x'$ because $\alpha$ is irrational. Let $P_{0} = P'$.
    Note that $x_0>x_{-1}$.
    We repeat this process starting with $P_{0}$, always choosing the new point with maximal $x$-coordinate.
    Let $m\geq 0$.
    We construct $P_{m+1}$ from $P_m$ by increasing $t$, starting at $t_m$, until $L_\alpha$ is no longer admissible for $\mathcal B_{t}(P_m)$, and set $t_{m+1}$ as the maximum.
    We call the point on the boundary of $\mathcal B_{t_{m+1}}(P_m)$ with maximal $x$-coordinate $P_{m+1}$.
    We show that the sequence of points constructed this way satisfies properties (1)--(4) of the lemma.
    
    (1) By Lemma~\ref{lemma:incr-decr-t} and our choice of $P_m$ with maximal $x$-coordinate we see that $P_m$ must be to the right of $P_{m-1}$, so $x_{m}>x_{m-1}$.
    Additionally, the boundary of the ball $\mathcal B_{t_m}(P_m)$ is concave, so $\abs{y_m}<\abs{y_{m-1}}$.
    
    (2)--(3) These follow from our choice of $t_m$ as the maximal $t$ for which $L_\alpha$ is admissible for $\mathcal B_{t}(P_{m-1})$ and our choice of $P_m$ on the boundary of $\mathcal B_{t_m}(P_{m-1})$.
    
    (4) For any $t\in (t_{m-1},t_m)$, the lattice $L_\alpha$ is admissible for $B_t(P_{m-1})$, so there are no points $P'\in L_\alpha$ with $F_t(P') < F_t(P_{m-1})$.
    If there were a point $P'\in L_\alpha$ for which $F_t(P')=F_t(P_{m-1})$ then this would contradict the definition of $t_m$.
\end{proof}

Figure~\ref{fig:ball plot} shows a few steps of the algorithm described above.

We claim that the sequence $\{t_m\}$ from Lemma~\ref{lem:algorithm} tends to infinity.
Indeed, the area of the ball $\mathcal B_t(P_m)$ is 
\begin{equation} \label{eq:area-stretched-ball}
    \area{\mathcal{B}_t(P_m)} = F_t^2(P_m) \text{ } \area{\mathcal{B}},
\end{equation} 
so we observe that
\begin{equation}
        x_m t_m^{-1} = F(x_m t_m^{-1}, 0) \leq F(x_m t_m^{-1}, y_m t_m) = F_{t_m}(P_m) \leq C_p^{\frac12} \area{\mathcal{B}}^{-\frac12},
\end{equation}
by Lemma \ref{lem:lower-norm} and Proposition \ref{prop:max-area}.
Since $C_p$ is fixed for any fixed $p$ and $x_m \to \infty$, we must have that $t_m \to \infty$.

\begin{figure}[t]
    \centering
    \includegraphics{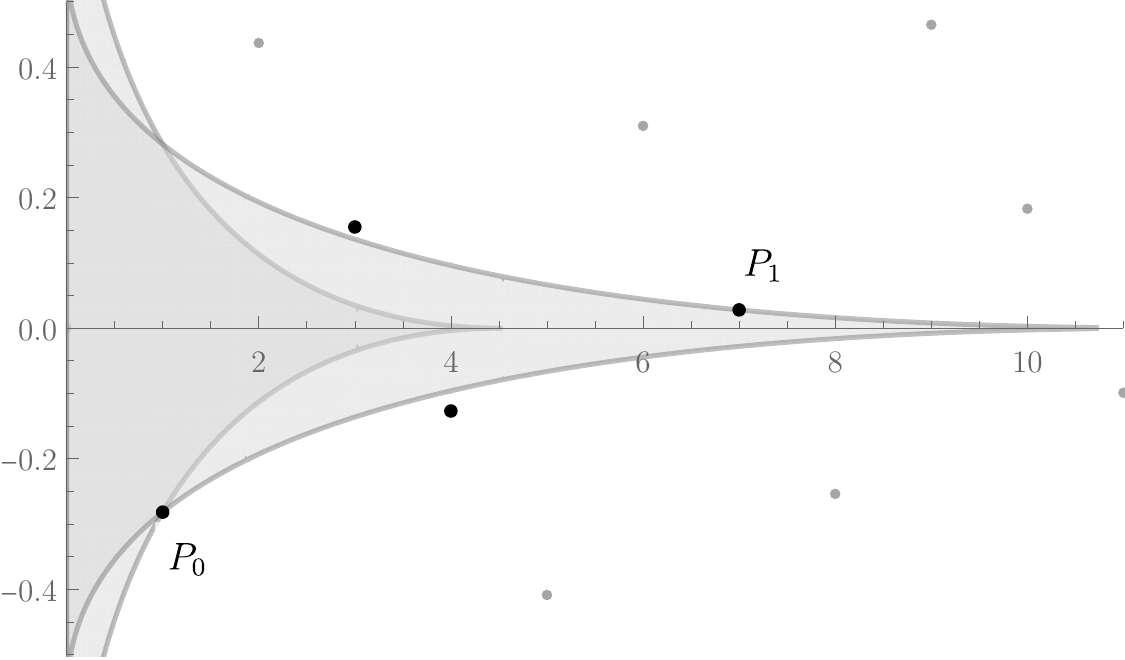}
    \caption{The lattice $L_\alpha$ for $\alpha = 3-e$ with $p=0.5$. Dark lattice points correspond to the regular convergents. The points $P_0 = (1,-\alpha)$ and $P_1=(7,2-7\alpha)$ give best approximations for $\alpha$ relative to $F^{\langle p \rangle}$.}
    \label{fig:ball plot}
\end{figure}

To any sequence $\{r_n/s_n\}$ of rational numbers we can associate a generalized continued fraction with rational coefficients; that is, an expression of the form
\begin{equation}
    a_0 + \frac{\ep_1}{a_1 + } \, \frac{\ep_2}{a_2 + } \, \frac{\ep_3}{a_3+} \cdots 
    := a_0 + \dfrac{\ep_1}{a_1+ \dfrac{\ep_2}{a_2+\dfrac{\ep_3}{a_3+\cdots}}}
\end{equation}
with $a_j,\ep_j\in \Q\setminus\{0\}$ and
\begin{equation}
    \frac{r_n}{s_n} = a_0 + \frac{\ep_1}{a_1 + } \, \frac{\ep_2}{a_2 + } \cdots \frac{\ep_n}{a_n}.
\end{equation}
The $\ep_j$ and $a_j$ are defined recursively in terms of $r_n$ and $s_n$, as we describe below.
If the sequence $\{r_n/s_n\}$ converges (resp.~diverges), we say that the continued fraction is convergent (resp.~divergent).
For any sequence $\{\rho_n\}$ of nonzero reals we have the transformation
\begin{equation} \label{eq:rho-transform}
    a_0 + \frac{\ep_1}{a_1 + } \, \frac{\ep_2}{a_2 + } \, \frac{\ep_3}{a_3+} \cdots = 
    a_0 + \frac{\rho_1\ep_1}{\rho_1 a_1 + } \, \frac{\rho_1\rho_2\ep_2}{\rho_2a_2 + } \, \frac{\rho_2\rho_3\ep_3}{\rho_3a_3+} \cdots
\end{equation}
which preserves the convergents, as can be shown by induction.
In particular, a continued fraction with rational coefficients can be easily transformed into one with integral coefficients (except possibly $a_0$) by clearing denominators. 
Of course, for a generic sequence, this process of clearing denominators will produce arbitrarily large values of $|\ep_j|$ and $|a_j|$.

Fix an irrational $\alpha\in (-\frac12, \frac12)$ and let $P_m = (x_m, y_m)$ as in Lemma $\ref{lem:algorithm}$. For $m \geq 1$, let \begin{equation}
    g_m = \begin{pmatrix} x_m & y_m \\ x_{m-1} & y_{m-1} \end{pmatrix}
\end{equation}
and set \begin{equation}g_{0} =\begin{pmatrix} x_0 & y_0 \\ 0 & x_0 \end{pmatrix}.\end{equation}
Note that $\det g_m \in \Z$ for all $m$. 
Define $\wt{a}_m$ and $\wt{\ep}_m$ by 
\begin{equation}\label{eq:a-tild-eps-tild-def}
    g_m = \begin{pmatrix}\wt{a}_m & \wt{\ep}_m \\ 1 & 0 \end{pmatrix}g_{m-1}
\end{equation} for $m \geq 1$.

We further define $s_m = x_m$ and $r_m = y_m + x_m \alpha$. Then, \begin{equation}
    g_m = 
    \begin{pmatrix} 
    s_m & r_m - s_m \alpha \\
    s_{m-1} & r_{m-1} - s_{m-1} \alpha \end{pmatrix}.
\end{equation}
We see that \begin{equation}
    \begin{pmatrix}
    r_m & s_m \\
    r_{m-1} & s_{m-1} \end{pmatrix}
    = g_m 
    \begin{pmatrix} 
    \alpha & 1 \\ 
    1 & 0 \end{pmatrix}.
\end{equation}
Substituting for $g_m$, we have that \begin{equation}
    \begin{pmatrix} r_m & s_m \\ r_{m-1} & s_{m-1} \end{pmatrix} = \begin{pmatrix}\wt{a}_m & \wt{\ep}_m \\ 1 & 0 \end{pmatrix} \begin{pmatrix} s_{m-1} & r_{m-1} - s_{m-1}\alpha \\ s_{m-2} & r_{m-2} - s_{m-2}\alpha \end{pmatrix} \begin{pmatrix} \alpha & 1 \\ 1 & 0 \end{pmatrix}.
\end{equation}
From this, we get the recurrence relations \begin{align}
    r_m &= \wt{a}_m r_{m-1} + \wt{\ep}_m r_{m-2}, & r_0 &= y_0 + x_0 \alpha, & r_{-1} &= x_0, \\
    s_m &= \wt{a}_m s_{m-1} + \wt{\ep}_m s_{m-2}, & s_0 &= x_0, & s_{-1} &= 0.
\end{align}
Following the analogous proof from the theory of regular continued fractions, we obtain \begin{equation}
    \frac{r_m}{s_m} = a_0 + \frac{\wt{\ep}_1}{\wt{a}_1 +} \, \frac{\wt{\ep}_2}{\wt{a}_2 +} \dots \frac{\wt{\ep}_m}{\wt{a}_m},
\end{equation} where $a_0 = r_0/s_0$.

We prefer to have integer coefficients.
To that end, we define 
\begin{equation}
    a_m = (\det{g_{m-1}}) \wt{a}_m
\end{equation}
for $m \geq 1$. For $m \geq 2$, we define 
\begin{equation}
    \ep_m = (\det{g_{m - 2}}) (\det{g_{m-1}}) \wt{\ep}_m
\end{equation}
and set $\ep_1 = (\det{g_0}) \wt{\ep_1}$. We note that the denominators of $\wt{a}_m$ and $\wt{\ep}_m$ are divisors of $\det{g_{m-1}}$ by their construction in \eqref{eq:a-tild-eps-tild-def}, so $a_m$ and $\ep_m$ are always integral.
We then have a generalized continued fraction for $r_m/s_m$, with only $a_0$ not strictly integral. We note that this transformation is one of the form given in \eqref{eq:rho-transform}, which preserves the convergents to the continued fraction. Thus,
\begin{equation}
    \frac{r_m}{s_m} = a_0 + \frac{\ep_1}{a_1 + } \, \frac{\ep_2}{a_2 + }\dots \frac{\ep_m}{a_m}.
\end{equation}
We are not guaranteed that $\gcd(\ep_m, a_m) = 1$ or even that $\gcd(\ep_m, a_m, \ep_{m+1}) = 1$. 
However, this generalized continued fraction is unique up to transformation.

\begin{proof}[Proof of Theorem~\ref{thm:main}]

Let $\delta > 0$. By Proposition \ref{thm:beta}, there exists a $p = p_\delta$ such that \[
4^{-\frac{1}{p}} \beta_p \leq \frac{1}{\sqrt{5}} + \delta.
\]

(1) The $p$-convergents are best approximations with respect to $F$ because the associated lattice points satisfy part (4) of Lemma \ref{lem:algorithm}.

(2) Since every $p$-convergent is a best approximation with respect to $F$, \eqref{eq:am-gm-beta} is satisfied. Thus, every $r_m/s_m$ satisfies \[
s_m \abs{r_m - s_m\alpha} < \frac{1}{\sqrt{5}} + \delta.
\]

(3) Since $\wt{\ep}_m = -\det{g_m}/\det{g_{m-1}}$, we have for $m \geq 2$ the bound \begin{equation}
\abs{\ep_m} \leq \abs{(\det g_{m-2}) (\det g_{m})}.
\end{equation}
For $m = 1$, we see that $\ep_1 \leq \abs{\det g_{m}}$.
We therefore only need to prove an upper bound on $\abs{\det g_m}$ to obtain a bound for $\abs{\ep_m}$.
This is accomplished in the next proposition, which proves that the determinant is bounded and the bound depends only on $p$. By applying a transformation as in \eqref{eq:rho-transform}, we could obtain an equivalent continued fraction where $\gcd(\ep_m, a_m, \ep_{m+1}) = 1$, but the values of each new $\abs{\ep_m}$ would only be smaller. These new partial numerators would still be bounded by a constant dependent only on $p$.
\end{proof}

\begin{proposition} \label{prop:index-upper-bd}
Fix $p\in (0,1)$ and let $\alpha \in (-\frac 12,\frac 12)$ be an irrational number.
With $g_m$ as above, for each $m\geq 1$ we have
\begin{equation}\label{eq:C_p}
    \abs{\det g_m} \leq 2^{4/p-1} \frac{\Gamma\big(1+\frac 1p)^2}{\Gamma\big(1+\frac 2p)} c_p^2.
\end{equation}
\end{proposition}

\begin{proof}
Without loss of generality we may assume that $y_{m-1}>0$.
We divide into two cases depending on the sign of $y_{m}$.

Suppose first that $y_{m}>0$.
Then the points $P_m = (x_{m},y_{m})$ and $P_{m-1} = (x_{m-1}, y_{m-1})$ are in the first quadrant, with $P_{m-1}$ to the left of $P_m$.
By Lemma~\ref{lem:algorithm} there exists a $t>0$ such that $F_t(P_m) = F_t(P_{m-1})$.
Let $Q_1 = (t^{-1}x_m,ty_{m}) = (u_1,v_1)$ and $Q_2 = (t^{-1}x_{m-1},ty_{m-1}) = (u_2,v_2)$.
Then $Q_1$ and $Q_2$ both lie on the boundary of $R\mathcal B$, where $R=F_t(P_m)$ and $\mathcal B=\{P\in \R^2:F(P)<1\}$ is the unit ball.  Furthermore,
\begin{equation}
    |\det g_m| = u_1 v_2 - u_2 v_1.
\end{equation}
We begin by computing the area of the region $D$ bounded by the line segments $OQ_1$ and $OQ_2$ and the portion of the curve $u^p+v^p=R^p$ in the first quadrant.
Let $\theta_1 = \arctan(v_1/u_1)$ and $\theta_2=\arctan(v_2/u_2)$ denote the angles that the line segments $OQ_1$ and $OQ_2$ make with the positive $x$-axis.
In polar coordinates $(r,\theta)$ the region $D$ is described as
\begin{equation}
    \theta_1 \leq \theta \leq \theta_2, \qquad 0\leq r\leq r(\theta) = \frac{R}{((\cos \theta)^p+(\sin \theta)^p)^{1/p}}.
\end{equation}
Thus we have
\begin{align}
    \operatorname{area}(D) = \int_{\theta_1}^{\theta_2} \int_0^{r(\theta)} r\, drd\theta = \frac {R^2}2 \int_{\theta_1}^{\theta_2} \frac{d\theta}{((\cos\theta)^p+(\sin\theta)^p)^{2/p}}.
\end{align}
Making the change of variable $u=\tan\theta$ we find that
\begin{align}
    \operatorname{area}(D) = \frac {R^2}2 \int_{v_1/u_1}^{v_2/u_2} \frac{du}{(1+u^p)^{2/p}} = \frac {R^2}2 \int_{v_1/u_1}^{v_2/u_2} \left(\frac{(1+u)^2}{(1+u^p)^{2/p}}\right)\frac{du}{(1+u)^2}.
\end{align}
The minimum of the function $u\mapsto \frac{(1+u)^2}{(1+u^p)^{2/p}}$ is $2^{2-2/p}$; it occurs at $u=1$.
Thus we have
\begin{align}
    \operatorname{area}(D) \geq 2^{1-2/p}R^2 \int_{v_1/u_1}^{v_2/u_2} \frac{du}{(1+u)^2} = 2^{1-2/p} R^2 \frac{u_1v_2-v_1u_2}{(u_1+v_1)(u_2+v_2)}.
\end{align}
The maximum value of $(u,v)\mapsto u+v$ subject to the constraint $u^p+v^p=R^p$ is $R$; it occurs when $(u,v)=(R,0)$. Thus we have
\begin{equation}
    \operatorname{area}(D) \geq 2^{1-2/p} |\det g_m|.
\end{equation}
There are two copies of $D$ inside $R\mathcal B = \mathcal B_1(Q_1)$, so we conclude by Proposition~\ref{prop:max-area} that
\begin{equation}
    |\det g_m| \leq 2^{2/p-2} (2\operatorname{area}(D)) \leq 2^{2/p-2} \operatorname{area}(\mathcal B_1(Q_1)) < 2^{4/p-1} \frac{\Gamma\big(1+\frac 1p\big)^2}{\Gamma\big(1+\frac 2p\big)} c_p^2.
\end{equation}

If $y_{m}<0$, define $Q_1$ and $Q_2$ as before and define $Q_1' = (u_1,-v_1)$.
Let $D$ denote the region bounded by the line segments $OQ_1'$ and $OQ_2$ and the portion of the curve $u^p+v^p=R^p$ in the first quadrant.
Then we have
\begin{equation}
    \operatorname{area}(D) \geq 2^{1-2/p}(u_1v_2+v_1u_2).
\end{equation}
Let $D'$ denote the triangle $OQ_1Q_1'$. 
Then
\begin{equation}
    \operatorname{area}(D') = u_1|v_1| \geq u_2|v_1| \geq 2^{2-2/p} u_2|v_1| = -2^{2-2/p} u_2v_1,
\end{equation}
so
\begin{equation}
    \operatorname{area}(D\cup D') \geq 2^{1-2/p} (u_1v_2 - v_1u_2) = 2^{1-2/p} |\det g_m|.
\end{equation}
The remainder of the proof proceeds as in the first case.
\end{proof}

\section{Regular convergents skipped by the \texorpdfstring{$p$}{p}-continued fraction}
\label{sec:skipped}

In this section we prove Theorem~\ref{thm:max-ell} and several lemmas to aid us in the next section, all of which concern the question of how many of the regular convergents of a given irrational $\alpha$ are skipped by the $p$-continued fraction.

We first list several well-known results about the regular continued fraction.
Fix an irrational $\alpha$.
For $p\in (0,1)$, let $p_n/q_n$ denote the regular convergents of $\alpha$ and $r_n/s_n$ the $p$-convergents. First, $p_n$ and $q_n$ are defined recursively by
\begin{align}\label{eq:p-q-rec}
    p_n = b_n p_{n-1} + p_{n-2}, && p_{-2} &= 0, & p_{-1} &= 1, \\
    q_n = b_n q_{n-1} + q_{n-2}, && q_{-2} &= 1, & q_{-1} &= 0,
\end{align} where $b_n$ is the $n$-th regular partial quotient.
Additionally, we have that for $n \geq -1$, 
\begin{equation}\label{eq:Schmidt-3C}
    q_n p_{n-1} - p_n q_{n-1} = (-1)^n,
\end{equation}
and for $n \geq 0$, 
\begin{equation}\label{eq:Schmidt-3D}
    q_np_{n-2} - p_nq_{n-2} = (-1)^{n-1}b_n.
\end{equation}
These results and their proofs can be found in Section~3 of \cite[Ch.~I]{schmidt}.

The number of consecutive regular convergents $p_n/q_n$ that can be skipped by the $p$-continued fraction is closely related to the determinants that appear in the following proposition.

\begin{proposition} \label{prop:D-ell}
Notation as above, let $b_n$ denote the $n$-th regular partial quotient for $\alpha$ and for $\ell\geq 0$ define
\begin{equation}
    D_\ell(n) = (-1)^{n-1}\left(q_{n-1} p_{n+\ell} - p_{n-1} q_{n+\ell}\right).
\end{equation}
Then $D_0(n) = 1$, $D_1(n) = b_{n+1}$, and for $\ell\geq 2$, we have
\begin{equation}
    D_{\ell}(n) = b_{n+\ell}D_{\ell-1}(n) + D_{\ell-2}(n).
\end{equation}
\end{proposition}

\begin{proof}

For the preliminary cases $\ell\in \{0,1\}$ we have $D_0(n) = 1$ and $D_1(n) = b_{n+1}$ by \eqref{eq:Schmidt-3C} and \eqref{eq:Schmidt-3D}.
For $\ell \geq 2$, using the recursive definitions of $p_{n+\ell}$ and $q_{n+\ell}$ we have that 
\begin{align}
    D_\ell(n) &= (-1)^{n-1}\left( q_{n-1}(b_{n+\ell} p_{n+\ell-1} + p_{n+\ell-2}) - p_{n-1} (b_{n+\ell} q_{n+\ell-1} + q_{n+\ell-2})\right) \\
    &= b_{n+\ell} (-1)^{n-1} (q_{n-1} p_{n+\ell-1} - p_{n-1} q_{n+\ell-1}) + (-1)^{n-1}(q_{n-1} p_{n+\ell-2} - p_{n-1} q_{n+\ell-2}) \\
    &= b_{n+\ell} D_{\ell-1}(n) + D_{\ell-2}(n),
\end{align}
as desired.
\end{proof}

\begin{proof}[Proof of Theorem~\ref{thm:max-ell}]
By Proposition~\ref{prop:D-ell}, using that the regular partial quotients of $\alpha$ satisfy $b_n\geq 1$ for all $n$, we have the rough lower bound 
\[
    D_\ell(n)\geq F_{\ell+1} \geq \frac{\varphi^\ell}{\sqrt{5}},
\]
where $F_\ell$ is the $\ell$-th Fibonacci number and $\varphi = \frac{1}{2}(1 + \sqrt{5})$.
On the other hand, $D_\ell(n) = \abs{\det g_m}$ for some $m$, so by Proposition~\ref{prop:index-upper-bd} we have $D_\ell(n) \leq C_p$ where $C_p$ is the expression on the right hand side in \eqref{eq:C_p}.
Thus
\begin{equation}
    \ell \leq \log_\varphi(\sqrt 5 C_p). \tag*{\qedhere}
\end{equation}
\end{proof}

To aid in proving Theorem~\ref{thm:1-per-big-skips}, we state several lemmas regarding 1-periodic continued fractions.
For the remainder of this section we assume that $a\in \Z$ with $a\geq 1$ and that
\begin{equation} \label{eq:alpha-1-per}
    \alpha = a+\frac{1}{a+}\,\frac{1}{a+}\,\frac{1}{a+}\cdots=\frac{a+\sqrt{a^2+4}}{2}.
\end{equation}

\begin{lemma}\label{lemma:R-seq}
    Suppose that $\alpha$ satisfies \eqref{eq:alpha-1-per} and define a sequence $R_n$ by $R_0 = 0$,  $R_1 = 1$, and
    \begin{equation}\label{eq:R-seq}
        R_n = a R_{n-1} + R_{n-2} \quad \text{ for } n\geq 2.
    \end{equation}
    Then for all $n\geq 0$ we have
    \begin{equation}
        R_n = \frac{\alpha^n - (-\alpha)^{-n}}{\alpha+\alpha^{-1}}.
    \end{equation}
\end{lemma}

\begin{proof}

We can rewrite \eqref{eq:R-seq} using matrices, letting us solve for any $R_{n+1}$ and $R_{n}$ for $n \geq 1$ by writing \begin{equation}
    \begin{pmatrix} R_{n+1} \\ R_{n} \end{pmatrix} = 
    \begin{pmatrix} a & 1 \\
    1 & 0 \end{pmatrix}^{n}
    \begin{pmatrix} 1 \\ 0 \end{pmatrix}.
\end{equation}
The eigenvalues of $\left(\begin{smallmatrix} a & 1 \\ 1 & 0 \end{smallmatrix} \right)$ are $\alpha$ and $-\alpha^{-1}$, where we have used $\eqref{eq:alpha-1-per}$ and $\alpha = a + \alpha^{-1}$. Diagonalizing, we obtain
\begin{equation}
    \begin{pmatrix} R_{n+1} \\ R_{n} \end{pmatrix} = 
    \frac{1}{\alpha + \alpha^{-1}}
    \begin{pmatrix} 1 & 1 \\
    \alpha^{-1} & -\alpha \end{pmatrix}
    \begin{pmatrix} \alpha & 0 \\
    0 & -\alpha^{-1} \end{pmatrix}^n
    \begin{pmatrix}  \alpha & 1 \\
    \alpha^{-1} & -1 \end{pmatrix}
    \begin{pmatrix} 1 \\ 0 \end{pmatrix}.
\end{equation}
Further simplification yields
\begin{equation}
    \begin{pmatrix} R_{n+1} \\ R_{n} \end{pmatrix} = 
    \frac{1}{\alpha + \alpha^{-1}}
    \begin{pmatrix} \alpha^{n+1} - (-\alpha)^{-n-1} & \alpha^{n} - (-\alpha)^{-n} \\
    \alpha^{n} - (-\alpha)^{-n} & \alpha^{n-1} - (-\alpha)^{-n+1}
    \end{pmatrix}
    \begin{pmatrix} 1 \\ 0 \end{pmatrix}.
\end{equation}
This completes the proof.
\end{proof}

\begin{lemma}\label{lemma:num-den-swap}
If $\alpha$ satisfies \eqref{eq:alpha-1-per} then $p_{n-1}=q_n$ for all $n\geq 1$.
\end{lemma}

\begin{proof}

Since $\alpha$ satisfies \eqref{eq:alpha-1-per}, every partial quotient is $a$.
Then because $q_0 = 1$, $p_n$ and $q_n$ each satisfy with offset the recursive definition for $R_n$ given in Lemma \ref{lemma:R-seq}. In particular, $p_n = R_{n+2}$ and $q_{n} = R_{n+1}$. Thus $p_{n-1} = R_{n+1} = q_n$.
\end{proof}

\begin{lemma}\label{lemma:theta}
    Suppose that $\alpha$ satisfies \eqref{eq:alpha-1-per} and let $R_n$ be as in Lemma~\ref{lemma:R-seq}.
    Then for any $n\geq 0$ we have
    \begin{equation}
        \abs{R_{n+1}-R_n\alpha} = \alpha^{-n}.
    \end{equation}
\end{lemma}

\begin{proof}
By Lemma \ref{lemma:R-seq},
\begin{align*}
    \abs{R_{n+1} - R_n\alpha}
    &= \abs{
        \frac{\alpha^{n+1}  - (-\alpha)^{-n-1}} {\alpha+\alpha^{-1}} -
        \frac{\alpha^n      - (-\alpha)^{-n}}     {\alpha+\alpha^{-1}}(\alpha)
    }
    \\
    &= \abs{ (-1)^{n} \frac{\alpha^{-n-1} + \alpha^{-n+1}}{\alpha+\alpha^{-1}} }
    \\
    &= \frac{\alpha^{-n-1} ( \alpha^{2} + 1 )}{\alpha+\alpha^{-1}}
    = \alpha^{-n},
\end{align*}
as desired.
\end{proof}

\section{Proof of Theorem \ref{thm:1-per-big-skips}}

\label{sec:1-per}

Fix an $\alpha$ that satisfies $\eqref{eq:alpha-1-per}$. We define a curve and examine its intersection with the lattice points corresponding to the regular convergents of $\alpha$. Let $Q_n = (R_n, \abs{R_{n+1}-R_n \alpha})$ for any $n \geq 0$. This point is either the lattice point corresponding to a regular convergent or the reflection of said point across the $x$-axis because of Lemma~\ref{lemma:num-den-swap} and our choice of $\alpha$. Fix an $m\in \N$. We define $\Cp$ to be the portion of the boundary of the ball $\Bp_t((0,1))$ in the first quadrant, where $t$ is such that $Q_m$ is on the boundary of $\Bp_t((0,1))$. We can describe the curve $\Cp$ algebraically by the equation
\begin{equation}\label{eq:C-alg}
    t^p =  (x t^{-1})^p + (yt)^p.
\end{equation}
Since $Q_m$ is on $\Cp$, we can write \begin{equation}\label{eq:t-def-C}
    t = \left( \frac{R_m^p}{1-\abs{R_{m+1}-R_m \alpha}^p} \right)^{\frac{1}{2p}}.
\end{equation}

We claim there exists a unique $p=p(m)$ such that $Q_{m+1}$ is also on $\Cp$. Equivalently,
\begin{equation}
    \left( \frac{R_m^p}{1-\abs{R_{m+1}-R_m \alpha}^p} \right)^{\frac{1}{2p}} = t = \left( \frac{R_{m+1}^p}{1-\abs{R_{m+2}-R_{m+1} \alpha}^p} \right)^{\frac{1}{2p}}.
\end{equation} Simplification yields
\begin{equation}
    R_{m}^p - R_{m}^p \abs{ R_{m+2} - R_{m+1} \alpha }^p =
    R_{m+1}^p - R_{m+1}^p \abs{ R_{m+1} - R_m \alpha }^p,
\end{equation}
or, after Lemma \ref{lemma:theta},
\begin{equation} \label{eq:p-def}
    \left(\frac{R_{m+1}}{R_m}\right)^p = \frac{1-\alpha^{-mp-p}}{1-\alpha^{-mp}}.
\end{equation}

\begin{lemma}\label{lemma:p-unique}
    If $m\geq 3$ then there exists a unique $p = p(m)\in (0,1)$ satisfying \eqref{eq:p-def}.
\end{lemma}

\begin{proof}
Let $m\geq 3$ be an integer and define
\begin{equation}
    f(p) = \frac{1-\alpha^{-mp-p}}{1-\alpha^{-mp}}.
\end{equation}
We claim that $f(p)$ is strictly decreasing for $p\in (0,\infty)$.
Assuming this for now, and using that $f(p)>1$ we find that
\begin{equation}
    \frac{d}{dp} \left[f(p)\right]^{\frac 1p} = \left(\frac{f'(p)}{p f(p)} - \frac{\log f(p)}{p^2}\right)\left[f(p)\right]^{\frac 1p} < 0.
\end{equation}
Since $\lim_{p\to 0^+} [f(p)]^{1/p} = \infty$ and $\lim_{p\to\infty}[f(p)]^{1/p} = 1$, we see that $p\mapsto [f(p)]^{1/p}$ is a bijection from $(0,\infty)$ to $(1,\infty)$.
Thus there is a unique $p$ for which $[f(p)]^{1/p} = R_{m+1}/R_m$.
We claim that this $p$ must be in $(0,1)$.
Indeed, when $p\geq 1$ we have
\begin{equation}
    [f(p)]^{\frac 1p} \leq f(1) = \frac{1-\alpha^{-m-1}}{1-\alpha^{-m}} = 1+\alpha^{-m}\frac{1-\alpha^{-1}}{1-\alpha^{-m}} < 1+\frac{1}{\alpha^3}.
\end{equation}
On the other hand, by Lemma~\ref{lemma:theta} we have
\begin{equation}
    \frac{R_{m+1}}{R_m} \geq \alpha - \frac{1}{\alpha^m R_m} \geq \alpha - \frac{1}{\alpha^3}.
\end{equation}
Since $x > 1+2/x^3$ for all $x>1.6$ and since $\alpha>1.618$, we see that $[f(p)]^{1/p} < R_{m+1}/R_m$,
so the $p$ satisfying \eqref{eq:p-def} must be in $(0,1)$.
Thus it remains to show that $f(p)$ is decreasing.

We have 
\begin{equation}
    f'(p) = -\frac{ \left[m(1-\alpha^{-p})-\alpha^{-p}(1-\alpha^{-mp})\right]\alpha^{-mp}\log\alpha }{(1-\alpha^{-mp})^2},
\end{equation}
so it is enough to show that
\begin{equation}
    m(1-\alpha^{-p})>\alpha^{-p}(1-\alpha^{-mp}).
\end{equation}
Factoring the polynomial $1-(\alpha^{-p})^m$ we find that
\begin{equation}
    \alpha^{-p}\frac{(1-\alpha^{-mp})}{(1-\alpha^{-p})} = \alpha^{-p}\left(1+\alpha^{-p}+\ldots+\alpha^{-p(m-1)}\right) <m.
\end{equation}
This completes the proof.
\end{proof}

We claim that for $p=p(m)$ as above with sufficiently large $m$, $L_\alpha$ is admissible for the ball $\Bp_t((0,1))$. This is equivalent to $Q_{m+k}$ being above $\Cp$ for any $k\in \Z\setminus\{0,1\}$ such that $k > -m$.
By solving \eqref{eq:C-alg} for $y$, and comparing with the $y$-coordinate of $Q_{m+k}$, we find that we can write the condition of $Q_{m+k}$ being above $\Cp$ as
\begin{equation}
    |R_{m+k+1} - R_{m+k}\alpha| > \left(1 - R_{m+k}^pt^{-2p}\right)^\frac{1}{p}.
\end{equation}
After some simplification and substituting the expression in \eqref{eq:t-def-C} for $t$, we see that for a fixed $m$, $Q_{m+k}$ remains above $\Cp$ if and only if 
\begin{align}
     R_{m+k}^p - R_{m+k}^p \abs{ R_{m+1} - R_m \alpha }^p > R_{m}^p - R_{m}^p \abs{ R_{m+k+1} - R_{m+k} \alpha }^p.
\end{align}
By Lemma \ref{lemma:theta}, this can be simplified to 
\begin{align} \label{eq:p-fib-k-ineq}
    \left(\frac{R_{m+k}}{R_m}\right)^p > \frac{1-\alpha^{-mp-kp}}{1-\alpha^{-mp}}.
\end{align}

\begin{figure}
    \centering
    \includegraphics{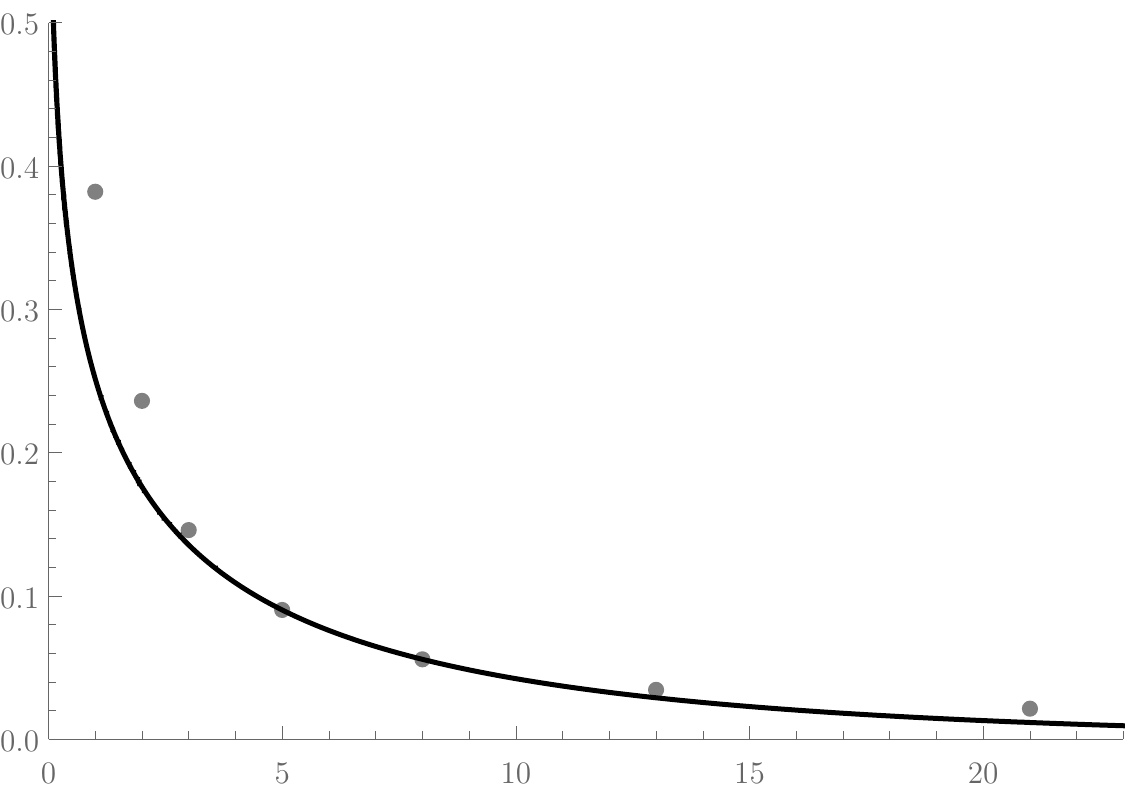}
    \caption{The edge of the ball, $\Cp$, for $p \approx .27$, with $\alpha = \varphi = \frac{1 + \sqrt{5}}{2}$ and $m = 5$. The points $Q_{m+k}$ for $k\in \{-3, -2, -1, 0, 1, 2, 3\}$ are all visible.}
    \label{fig:fib plot}
\end{figure}

\begin{lemma}\label{lemma:admissible}
    For each integer $m\geq 2$, define $p=p(m)\in(0,1)$ by \eqref{eq:p-def}.
    Then $p\sim \frac {\log 2}{m\log\alpha}$ as $m\to\infty$. Furthermore, for all sufficiently large $m$,
    for any $k\in \Z\setminus \{0,1\}$ with $k > -m$, the inequality \eqref{eq:p-fib-k-ineq} holds.
\end{lemma}

\begin{proof}
Let $k\in \Z\setminus\{0\}$.
Using Lemma \ref{lemma:R-seq} we have
\begin{align}
    \frac{R_{m+k}}{R_m} 
    &= \alpha^k\frac{1-(-1)^{m+k}\alpha^{-2m-2k}}{1-(-1)^{m}\alpha^{-2m}} \\
    &= \alpha^k \left(1 + (-1)^m \alpha^{-2m}\frac{1-(-1)^k\alpha^{-2k}}{1-(-1)^m \alpha^{-2m}}\right) \\
    &= \alpha^k\left(1 + O(\alpha^{-2m})\right).
\end{align}
Thus we have
\begin{equation} \label{eq:fib-ratio-k-asymp}
    \log\left(\frac{R_{m+k}}{R_m}\right) = k\log \alpha + O\left(\alpha^{-2m}\right).
\end{equation}

Making the ansatz
\begin{equation}
    p = \frac{p_0}{m} + \frac{p_1}{m^2} + \frac{p_2}{m^3} + O\left(\frac{1}{m^4}\right),
\end{equation}
we find that
\begin{align}
    \alpha^{-mp} &= \alpha^{-p_0-p_1m^{-1}-p_2m^{-2}+O(m^{-3})} \\
    &= \alpha^{-p_0} - p_1 \frac{\alpha^{-p_0}\log\alpha}{m} + \frac{\alpha^{-p_0}\log\alpha(p_1^2\log\alpha-2p_2)}{2m^2} + O\left(\frac{1}{m^3}\right),
\end{align}
and thus
\begin{multline} \label{eq:phi-p-asymp}
    \log\left(\frac{1-\alpha^{-mp-kp}}{1-\alpha^{-mp}}\right) \\
    =  \frac{k p_0\log\alpha}{(\alpha^{p_0}-1)m}
    - \frac{k\log \alpha \left((k p_0^2+2 p_0p_1) \alpha^{p_0} \log \alpha-2p_1
   \left(\alpha^{p_0}-1\right)\right)}{2 \left(\alpha^{p_0}-1\right)^2 m^2} +
    O\left(\frac{1}{m^3}\right).
\end{multline}
Hence, by the logarithm of \eqref{eq:p-def}, \eqref{eq:fib-ratio-k-asymp}, and \eqref{eq:phi-p-asymp} with $k=1$ we must have
\begin{equation}
    \frac{p_0 \log\alpha}{m} = \frac{p_0\log\alpha}{(\alpha^{p_0}-1)m}
\end{equation}
and
\begin{equation}
    \frac{p_1\log \alpha}{m^2} =  - \frac{\log \alpha \left((p_0^2+2 p_0p_1) \alpha^{p_0} \log \alpha-2p_1
   \left(\alpha^{p_0}-1\right)\right)}{2 \left(\alpha^{p_0}-1\right)^2 m^2}.
\end{equation}
The first equation shows that $\alpha^{p_0}-1=1$, which shows that $p_0 = \frac{\log 2}{\log \alpha}$ and simplifies the second equation significantly.
Thus
\begin{equation} \label{eq:p-m-asymp}
    p = \frac{\log 2}{m\log\alpha}\left( 1 - \frac{1}{2m} \right) + O\left(\frac{1}{m^3}\right).
\end{equation}

Now we assume that $k\neq 1$.
Using \eqref{eq:fib-ratio-k-asymp}, \eqref{eq:phi-p-asymp}, and \eqref{eq:p-m-asymp}, the logarithm of the inequality \eqref{eq:p-fib-k-ineq} becomes
\begin{equation}
    \frac{k\log 2}{m} - \frac{k\log 2}{2m^2} + O\left(\frac{1}{m^3}\right) > \frac{k\log 2}{m} - \frac{k\log 2}{2m^2}\left(2k\log 2 + 1-2\log 2\right) + O\left(\frac{1}{m^3}\right),
\end{equation}
which, for sufficiently large $m$, is true for any $k\in \Z\setminus \{0,1\}$.
\end{proof}

For any sufficiently large $m$ and with $p=p(m)$, the first two lattice points that touch the boundary of $\Bp_t((0,1))$ as $t$ increases from $1$ must be the lattice points corresponding to $Q_m$ and $Q_{m+1}$ by the previous lemma, and furthermore the lattice must be admissible for this ball. This then gives us that $r_0/s_0 = R_{m+2}/R_{m+1}=p_m/q_m$.
\qed

\bibliographystyle{plain}
\bibliography{biblio.bib}

\end{document}